\documentclass[12pt,leqno]{amsart}
\usepackage{amsmath,amssymb,amscd,latexsym,amsthm,mathrsfs,comment}
\usepackage[unicode]{hyperref}
\usepackage{hypbmsec}
\ifx\pdfoutput\undefined\else
\hypersetup{
colorlinks=true,
linkcolor=blue,
citecolor=blue,
urlcolor=blue,
filecolor=blue,
bookmarksnumbered=true,
pdfstartview=FitH,
pdfhighlight=/N
}
\fi
\newtheorem{theorem}{Theorem}
\newtheorem{lemma}{Lemma}

\newtheorem{proposition}{Proposition}
\theoremstyle{remark}
\newtheorem{remark}{Remark}

\numberwithin{equation}{section}

\def\Im{\operatorname{Im}}
\def\Li{\operatorname{Li}}
\def\ord{\operatorname{ord}}

\def\la{\lambda}

\def\al{\alpha}
\def\ep{\varepsilon}

\newcommand{\ZZ}{\mathbb Z}
\renewcommand\pmod[1]{\;(\operatorname{mod}#1)}

\newcommand{\wt}{\widetilde}

\newcommand\qbinom[3]{\begin{bmatrix}#1\\#2\end{bmatrix}_{#3}}
\newcommand\qsbinom[3]{\bigl[\begin{smallmatrix}#1\\#2\end{smallmatrix}\bigr]_{#3}}
\def\fl#1{\left\lfloor#1\right\rfloor}
\def\cl#1{\left\lceil#1\right\rceil}
\newcommand{\cD}{\mathcal D}
\newcommand{\tcD}{\widetilde{\mathcal D}}
\newcommand{\cN}{\mathcal N}

\def\Ddots{\vbox to12pt{\vss\hbox{\kern1pt\lower1pt\hbox to4pt{$\cdot$\hss}%
 \lower-2pt\hbox to4pt{\hss$\cdot$\hss}%
 \lower-5pt\hbox to4pt{\hss$\cdot$}}}}

\begin{document}

\hypersetup{pdfauthor={Christian Krattenthaler, Igor Rochev, Keijo V\"a\"an\"anen, and Wadim Zudilin},%
pdftitle={On the non-quadraticity of values of the \$q\$-exponential function and
related \$q\$-series}}

\title[On the non-quadraticity of values of $q$-series]%
{On the non-quadraticity of values\\of the $q$-exponential function and
related $q$-series}

\dedicatory{To our great Peter Bundschuh on his 70th birthday}

\author[C.~Krattenthaler]{Christian Krattenthaler}
\address{Fakult\"at f\"ur Mathematik, Universit\"at Wien, Nordbergstrasse 15,
A-1090 Vienna, AUSTRIA\newline\hbox to\parindent{\hss}%
\textit{WWW-address}: \href{http://www.mat.univie.ac.at/~kratt}{\tt http://www.mat.univie.ac.at/\~{}kratt}}

\author[I.~Rochev]{Igor Rochev}
\address{Department of Mechanics and Mathematics, Moscow Lomonosov
State University, Vorobiovy Gory, GSP-1, 119991 Moscow, RUSSIA}
\email{justrip@rambler.ru}

\author[K.~V\"a\"an\"anen]{Keijo V\"a\"an\"anen}
\address{Department of Mathematical Sciences, University of Oulu, P.\,O.~Box 3000,
90014 Oulu, FINLAND}
\email{kvaanane@sun3.oulu.fi}

\author[W.~Zudilin]{Wadim~Zudilin}
\address{Max-Planck-Institut f\"ur Mathematik, Vivatsgasse 7,
D-53111 Bonn, GERMANY%
\newline\hbox to\parindent{\hss}%
\textit{WWW-address}: \href{http://wain.mi.ras.ru/}{\tt http://wain.mi.ras.ru/}}

\date{August 12, 2008}

\thanks{The work of the first author was partially supported by the Austrian
Science Foundation FWF, grant S9607-N13,
in the framework of the National Research Network
``Analytic Combinatorics and Probabilistic Number Theory.''
The work of the second author was supported by
the Russian Foundation for Basic Research, grant no.~06-01-00518.
The work of the fourth author was supported by a fellowship
of the Max Planck Institute for Mathematics (Bonn). Part of this work
was done during the first and fourth authors' stay at the
Erwin Schr\"odinger Institute for Physics and Mathematics,
Vienna, during the programme ``Combinatorics and Statistical Physics''
in Spring~2008.}

\begin{abstract}
We investigate arithmetic properties of values of the entire function
$$
F(z)=F_q(z;\lambda)=\sum_{n=0}^\infty\frac{z^n}{\prod_{j=1}^n(q^j-\lambda)},
\qquad |q|>1, \quad \lambda\notin q^{\mathbb Z_{>0}},
$$
that includes as special cases the Tschakaloff function ($\lambda=0$)
and the $q$-exp\-onential function ($\lambda=1$). In particular,
we prove the non-quadraticity of the numbers $F_q(\alpha;\lambda)$ for
integral~$q$, rational $\lambda$ and $\alpha\notin-\lambda q^{\mathbb Z_{>0}}$, $\alpha\ne0$.
\end{abstract}

\maketitle

\section{Introduction and main results}
\label{s1}

Consider the $q$-exponential function
\begin{equation}
\label{e01}
E_q(z)=\sum_{n=0}^\infty\frac{z^n}{\prod_{j=1}^n(q^j-1)},
\end{equation}
which is an entire function in the complex $z$-plane for any $q\in\mathbb C$, $|q|>1$.
It is not difficult to adopt the classical proof of the irrationality of
$$
e=\sum_{n=0}^\infty\frac1{n!}
$$
to the case of the number $E_q(1)$ for an integer $q>1$.
Indeed, assuming, by contradiction, that $E_q(1)=r/s$~for certain positive integers $r$ and~$s$,
we see that the real number
\begin{multline}
\label{e02}
r\prod_{j=1}^k(q^j-1)-s\sum_{n=0}^k\prod_{j=n+1}^k(q^j-1)
\\
=s\prod_{j=1}^k(q^j-1)\cdot\biggl(E_q(1)-\sum_{n=0}^k\frac1{\prod_{j=1}^n(q^j-1)}\biggr)
=s\sum_{n=k+1}^\infty\frac1{\prod_{j=k+1}^n(q^j-1)}
\end{multline}
is integral (according to the left-hand side representation)  and
positive (because of the right-hand side representation),
hence it is at least~$1$, for any integer $k\ge1$.
On the other hand,
\begin{align*}
s\sum_{n=k+1}^\infty\frac1{\prod_{j=k+1}^n(q^j-1)}
&<\frac s{q^{k+1}-1}\sum_{n=0}^\infty\frac1{2^n}
\\
&=\frac{2s}{q^{k+1}-1}\to0
\quad\text{as}\; k\to\infty,
\end{align*}
leading to a contradiction.

The above proof is based on the simple observation that truncations
of the series defining $E_q(1)$ (see the intermediate term in~\eqref{e02})
provide rational approximations that are good enough
to conclude the irrationality of the number in question.
This argument has been generalized in various ways. For example, this
truncation idea lies at the heart of Mahler's method~\cite{Ma}
of proving the algebraic independence of values of the series
satisfying certain, quite restrictive, functional equations.
In the same paper~\cite{Ma}, K.~Mahler posed a transcendence
problem for values of the series that form a solution to more general
functional equations. This problem remains unsolved until today,
with the sole exception of values of quasi-modular functions~\cite{Ne}.
In particular, only irrationality and linear independence results
are known so far for values of the $q$-exponential function.

Recently, J.-P.~B\'ezivin~\cite{Be} proposed a new approach for
the study of arithmetic properties of values of certain $q$-series.
Among other things, he managed to prove the non-quadraticity
of values of the so-called Tschakaloff function
\begin{equation}
\label{e03}
T_q(z)=\sum_{n=0}^\infty q^{-n(n+1)/2}z^n
\end{equation}
at non-zero rational points if $q=\rho/\sigma\in\mathbb{Q}$ satisfies
$\gamma:=\log|\rho|/\log|\sigma|>14$. Furthermore, he proved the irrationality
of these values if $\gamma>28/15=1.866\dots$, and thus extended considerably
the possible values of~$q$ in the earlier irrationality results~\cite{Tsch},
\cite{Bu1}--\cite{Bu4}, where $\gamma>(3+\sqrt{5})/2=2.618\dots$\,.
It is interesting that B\'ezivin's approach was also an implicit
generalization of the truncation idea. The method of~\cite{Be}
was applied to the $q$-exponential function by R.~Choulet~\cite{Ch},
who could not prove the non-quadraticity of its values,
but improved the bound $\gamma>7/3$ of the earlier irrationality
result of Bundschuh \cite{Bu1} for $E_q(z)$ to $\gamma>2$.
He also improved
the above bound $\gamma>14$ in B\'ezivin's non-quadraticity result
for $T_q(z)$ to $\gamma>14/3$ and
the bound $\gamma>28/15$ in the irrationality result to $\gamma>28/17$.

The aim of this article is two-fold. First of all, we further generalize
B\'ezivin's method~\cite{Be} to prove non-quadraticity results for values of the $q$-series
\begin{equation}
\label{e04}
F(z)=F_q(z;\lambda)=\sum_{n=0}^\infty\frac{z^n}{\prod_{j=1}^n(q^j-\lambda)},
\qquad |q|>1,
\end{equation}
that include the Tschakaloff function and the $q$-exponential function as special
cases ($\lambda=0$ and $\lambda=1$, respectively),
and we further extend the values of~$q$ giving irrational values for $F_q(z;\lambda)$.
Secondly, in our proofs we use a more direct method than the $p$-adic approach
used in~\cite{Be} and~\cite{Ch}.
This allows us to perceive the additional
arithmetic information which can hardly be seen from the $p$-adic considerations.

We state our results in the following two theorems.

\begin{theorem}
\label{t1}
Let $q=\rho/\sigma\in\mathbb Q$ with $|q|>1$, and let
$\alpha$ and $\lambda$ satisfy $\alpha\ne0$, $\lambda\notin q^{\mathbb Z_{>0}}$
and $\alpha\notin-\lambda q^{\mathbb Z_{>0}}$.
If
\begin{equation*}
\gamma=\frac{\log|\rho|}{\log|\sigma|}>\begin{cases}
\dfrac {126 {{\pi }^2}} {47 {{\pi }^2}-72 {\sqrt{3}} \Im \Li_2(e^{2\pi
\sqrt{-1}/3}) }=3.27694460\dots &\text{if $\lambda=0$}, \\[2mm]
\dfrac {27 {{\pi }^2}} {5 {{\pi }^2}-18 {\sqrt{3}} \Im \Li_2(e^{2\pi \sqrt{-1}/3})}
=9.43194241\dots &\text{if $\lambda\ne0$},
\end{cases}
\end{equation*}
then $\alpha$, $\lambda$, and $\mu=F_q(\alpha;\lambda)$ in~\eqref{e04}
cannot all belong to a quadratic extension of~$\mathbb Q$. In particular,
if $\alpha$ and~$\lambda$ are rational then $F_q(\alpha;\lambda)$ is neither
rational nor quadratic.
\end{theorem}

In the case that $\lambda\ne0$, the above result is entirely new,
while its special case $\lambda=0$
improves Choulet's bound $\gamma>14/3$ considerably.

The next theorem gives
improvements for the above mentioned lower bounds of~$\gamma$ in the irrationality results.

\begin{theorem}
\label{t2}
Under the hypotheses of Theorem~\textup{\ref{t1}}, if
\begin{equation*}
\gamma=\frac{\log|\rho|}{\log|\sigma|}>\begin{cases}
\dfrac {252 {{\pi }^2}} {173 {{\pi }^2}-72 {\sqrt{3}}
\Im \Li_2(e^{2\pi \sqrt{-1}/3}) }
=1.53237645\dots &\text{if $\lambda=0$}, \\[2mm]
\dfrac {27 {{\pi }^2}} {16 {{\pi }^2}-9 {\sqrt{3}}
\Im \Li_2(e^{2\pi \sqrt{-1}/3}) }
=1.80828115\dots &\text{if $\lambda\ne0$},
\end{cases}
\end{equation*}
then $\alpha$, $\lambda$, and $\mu=F_q(\alpha;\lambda)$ in~\eqref{e04}
cannot all be rational.
\end{theorem}

Since the function $F_q(z;\lambda)$ satisfies the functional equation
$$F(qz)=(z+\lambda)F(z)+(1-\lambda),$$
the irrationality of the values of
$F_q(z;\lambda)$ at non-zero rational points $\notin -\lambda q^{\mathbb{Z}_{>0}}$
follows from~\cite{Sti} if a rational
number $\lambda \notin q^{\mathbb{Z}_{>0}}$
and a rational number $q$ satisfies $\gamma>7/3$ for
$\lambda \in q^{\mathbb{Z}_{\le 0}}$ and $\gamma>2+\sqrt{2}$ otherwise.

\medskip
Sections~\ref{s2}--\ref{s4} prepare for the proofs of these theorems.
In Section~\ref{s2}, we review B\'ezivin's construction, applied to
our more general context.
It involves in particular the introduction of a sequence
$(v_n)_{n\in\ZZ}$, the Hankel determinant of which plays a fundamental
role in the sequel. This determinant is a
polynomial in $q$ and two other variables.
Propositions~\ref{IRP1} and \ref{IRP2} in Section~\ref{s3} address the
power of $q$ which appears in this Hankel
determinant as a polynomial factor, while an asymptotic upper bound
for the Hankel determinant is found in Proposition~\ref{IRP3}.
Finally, Proposition~\ref{l0} in Section~\ref{s4} detects large
amounts of cyclotomic factors (in $q$) in the Hankel determinant.
All these ingredients are put together for the proofs of
Theorems~\ref{t1} and~\ref{t2} in Section~\ref{s5}.

\section{Review of B\'ezivin's construction}
\label{s2}

The general idea of B\'ezivin's method \cite{Be} refers to a function
\begin{equation}
\label{e05}
F(z)=\sum_{n=0}^\infty a_n(q)z^n,
\qquad a_0(q)=1, \quad
\frac{a_{n-1}(q)}{a_n(q)}=b_n(q)=b(q^n)
\;\;\text{for}\; n=1,2,\dots,
\end{equation}
where $b(\,\cdot\,)$~is a polynomial (in general, a rational function)
over a number field. Let $\alpha\in\mathbb C$.
One takes the coefficients $v_n$ appearing in
\begin{equation}
\label{e06}
\frac{F(\alpha z)-F(\alpha)}{z-1}
=\sum_{n=0}^\infty v_na_n(q)z^n
\end{equation}
and forms the Hankel determinant
\begin{equation}
\label{e07}
V_n=\det_{0\le i,j\le n-1}(v_{i+j}).
\end{equation}
Then one has to provide an analytic upper bound for $|V_n|$
and, under the assumption that both $\alpha$ and $\mu=F(\alpha)$
belong to a certain algebraic number field~$K$, an arithmetic lower
bound, in order to find them contradictory; this shows that the
assumption on~$\alpha$ and $\mu$ cannot be true.

Before going into the details of the construction, note that
relation~\eqref{e06} may be written in the form
$$
\sum_{n=0}^\infty a_n(q)\alpha^nz^n-\mu
=(z-1)\sum_{n=0}^\infty v_na_n(q)z^n
=-v_0+\sum_{n=1}^\infty\bigl(v_{n-1}a_{n-1}(q)-v_na_n(q)\bigr)z^n,
$$
yielding
\begin{equation}
\label{e08}
v_0=\mu-1 \qquad\text{and}\qquad
v_n=v_{n-1}b_n(q)-\alpha^n \quad\text{for}\; n=1,2,\dots.
\end{equation}
Hence, by induction, we easily arrive at the formula
\begin{equation}
\label{e09}
v_n=\mu\prod_{j=1}^nb_j(q)-\sum_{k=0}^n\alpha^k\prod_{j=k+1}^nb_j(q).
\end{equation}

\begin{remark} \label{rem:1}
Since we shall make use of it later on,
we point out that Formula~\eqref{e09} also holds for {\it negative}
$n$ (that is, if we extend the sequence $(v_n)$ to all integers $n$
by letting the recurrence \eqref{e08} hold for all integers $n$)
under the conventions
$$\sum _{k=m} ^{n-1}\operatorname{Expr}(k)=\begin{cases} \hphantom{-}
\sum _{k=m} ^{n-1} \operatorname{Expr}(k)&n>m,\\
\hphantom{-}0&n=m\\
-\sum _{k=n} ^{m-1}\operatorname{Expr}(k)&n<m,\end{cases}
$$
and
$$\prod _{k=m} ^{n-1}\operatorname{Expr}(k)=\begin{cases} \hphantom{-}
\prod _{k=m} ^{n-1} \operatorname{Expr}(k)&n>m,\\
\hphantom{-}1&n=m\\
1\Big/\prod _{k=n} ^{m-1}\operatorname{Expr}(k)&n<m.\end{cases}
$$
\end{remark}

Assuming that $b(\,\cdot\,)$ in~\eqref{e05} is a polynomial of degree~$s$,
Formula~\eqref{e09} shows that, for positive integers $n$,
$V_n$~is a polynomial in $\mu$, $\alpha$, and $q$
of degree at most $n$ in~$\mu$, $n(n-1)$ in~$\alpha$, and
\begin{equation} \label{e10}
s\sum_{i=0}^{n-1}\frac{2i(2i+1)}2=\frac{sn(n-1)(4n+1)}6
\end{equation}
in~$q$ (cf.\ \cite[Lemma~2.4]{Be}).
Formula~\eqref{e09} may also be written as
\begin{align*}
v_n
&=\prod_{j=1}^nb_j(q)\cdot\biggl(\mu-\sum_{k=0}^n\alpha^k\prod_{j=1}^k\frac1{b_j(q)}\biggr)
\\
&=a_n(q)^{-1}\cdot\biggl(\sum_{k=0}^\infty a_k(q)\alpha^k-\sum_{k=0}^na_k(q)\alpha^k\biggr)
\\
&=a_n(q)^{-1}\cdot\sum_{k=n+1}^\infty a_k(q)\alpha^k
=\sum_{k=n+1}^\infty\frac{\alpha^k}{\prod_{j=n+1}^kb_j(q)},
\end{align*}
showing that the $v_n$'s are nothing else but \emph{tails} of the series
$\mu=\sum_{k=0}^\infty a_k(q)\alpha^k$ (normalized by the factors $a_n(q)^{-1}$;
cf.\ the intermediate part of~\eqref{e02}).
This fact somehow explains why the determinant in~\eqref{e07} is
expected to be `small'.

\smallskip
Our basic example~\eqref{e04} corresponds to the choice
$b_n(q)=q^n-\lambda$, for
a fixed algebraic number~$\lambda$. In this case, we have
$a_n(q)=\prod_{k=1}^n(q^k-\lambda)^{-1}$, and
the Hankel determinant $V_n$~is also a polynomial
in~$\lambda$ of degree at most $n(n-1)$. The choice $b_n(q)=q^n$ (that is,
$\lambda=0$), yielding the Tschakaloff function~\eqref{e03},
was the illustrative example of the method in~\cite{Be}, while
the choice $b_n(q)=q^n-1$ (when $\lambda=1$) results in the $q$-exponential
function~\eqref{e01}. In~\cite{Ch}, Choulet treated both the
Tschakaloff and $q$-exponential cases.

We replace the argument of B\'ezivin and Choulet by a more direct approach
(see Sections~\ref{s3}--\ref{s5} below); in particular, we do not require
the non-trivial $p$-adic techniques used in~\cite{Be} and~\cite{Ch},
thus making our proofs more `concrete' and elementary. An essential gain,
which allows us to succeed in proving the non-quadraticity of the values of~\eqref{e04},
is due to extraction of cyclotomic factors in the factorization of
the Hankel determinant~\eqref{e07}; this is explained in Section~\ref{s4}.

\section{Determinant calculus}
\label{s3}

Define the ($q$-)\emph{order} of a Laurent series $f(q)=\sum_{n\in\mathbb Z}c_nq^n$ as
$$
\ord f(q)=\ord_qf(q)=\min\{n:c_n\ne0\}.
$$
The $q$-\emph{binomial coefficient} $\qsbinom mkq$ is defined by
\begin{equation*}
\qbinom mkq=\begin{cases}
\dfrac{(1-q^m)(1-q^{m-1})\cdots(1-q^{m-k+1})}
{(1-q^k)(1-q^{k-1})\cdots(1-q)} &\text{if $k\ge0$}, \\
0 &\text{if $k<0$}.
\end{cases}
\end{equation*}
Moreover, we adopt the usual notation for shifted $q$-factorials,
given by $(a;q)_m:=(1-a)(1-aq)\cdots(1-aq^{m-1})$ if $m>0$, and
$(a;q)_0:=1$.

Specializing $b_j(q)=q^j-\la$ in \eqref{e08},
where $\la\ne q^{\mathbb Z_{>0}}$, we
consider the sequence defined by
\begin{equation}
v_0=\mu-1, \qquad v_n=(q^n-\lambda)v_{n-1}-\alpha^n ,
\label{e11}
\end{equation}
where
\begin{equation} \label{eq:mu}
\mu=\sum _{n=0} ^{\infty}\frac {\al^n}
{\prod _{k=1} ^{n}(q^k-\la)}.
\end{equation}
We follow Remark~\ref{rem:1} in
requiring the recursive relation to be valid for all $n\in\mathbb Z$.
This does, in fact, not work if $q^n-\la=0$ for some integer $n\le0$.
However, since the only places where we take recourse on the extension of
\eqref{e11} to negative integers is in Remark~\ref{IRR1} and in the
proof of Proposition~\ref{IRP2}, in a
context where $\la=0$, we do not have to worry about these exceptional
cases.

Let $\cN$~denote the backward shift operator
acting (solely) on the index of the sequence $(v_n)_{n\in\mathbb Z}$,
that is $\cN v_n=v_{n-1}$.
Introduce the difference operator
\begin{equation}
\cD_l=(-\lambda\cN;q)_l\,(\alpha\cN;q)_l
=\prod_{k=0}^{l-1}(\mathcal I+(\lambda-\alpha)q^k\cN-\lambda\alpha q^{2k}\cN^2),
\label{e12}
\end{equation}
where $\mathcal I$ is the identity operator.

\begin{lemma}
\label{IR1}
For $n\in\mathbb Z$ and $l\ge0$ we have
\begin{equation}
\cD_lv_n=q^{l(n-l)}\sum_{s=0}^l\qbinom lsq
q^{\binom {l-s+1}2}(-\alpha)^{s}v_{n-l-s}.
\label{e13}
\end{equation}
\end{lemma}

\begin{proof}
By the $q$-binomial theorem (cf.\ \cite[Ex.~1.2(vi)]{GR})
\begin{equation} \label{eq:qbin}
(1+z)(1+qz)\cdots(1+q^{m-1}z)=
\sum _{\ell=0} ^{m}q^{\binom \ell2}\qbinom m\ell q z^\ell,
\end{equation}
we can write
$$
\cD_l=\sum _{k_1=0} ^{l}\sum _{k_2=0} ^{l}
q^{\binom {k_1}2+\binom {k_2}2}\qbinom l{k_1}q
\qbinom l{k_2}q \la^{k_1}(-\alpha)^{k_2}
\cN^{k_1+k_2}.
$$
Hence,
what we want to prove is
\begin{multline}\label{1}
    \sum_{k_1=0}^l\sum_{k_2=0}^lq^{\binom{k_1}2+\binom{k_2}2}\qbinom
    l{k_1}q\qbinom
    l{k_2}q\lambda^{k_1}(-\alpha)^{k_2}v_{n-k_1-k_2}\\
=q^{l(n-l)}\sum_{s=0}^l\qbinom
    lsqq^{\binom{l-s+1}2}(-\alpha)^{s}v_{n-l-s}
\end{multline}
for $n\in\mathbb Z$ and $l\in\mathbb N_0$. For $l=0,1$, this equality
can be readily verified.

We now assume that \eqref{1} is valid for some $l\ge1$ and all $n$.
Substituting $n-1$ and $n-2$ instead of~$n$, we get
\begin{multline}\label{2}
    \sum_{k_1=0}^l\sum_{k_2=0}^lq^{\binom{k_1}2+\binom{k_2}2}\qbinom
    l{k_1}q\qbinom
    l{k_2}q\lambda^{k_1}(-\alpha)^{k_2}v_{n-k_1-k_2-1}\\
=q^{l(n-l-1)}\sum_{s=0}^l\qbinom
    lsqq^{\binom{l-s+1}2}(-\alpha)^{s}v_{n-l-s-1},
\end{multline}
respectively
\begin{multline}\label{3}
    \sum_{k_1=0}^l\sum_{k_2=0}^lq^{\binom{k_1}2+\binom{k_2}2}\qbinom
    l{k_1}q\qbinom
    l{k_2}q\lambda^{k_1}(-\alpha)^{k_2}v_{n-k_1-k_2-2}\\
=q^{l(n-l-2)}\sum_{s=0}^l\qbinom
    lsqq^{\binom{l-s+1}2}(-\alpha)^{s}v_{n-l-s-2}.
\end{multline}
Next we form the linear combination
\begin{equation} \label{eq:lincomb}
\eqref{1}+(\lambda-\alpha)q^l\cdot\eqref{2}-\lambda\alpha
q^{2l}\cdot\eqref{3}.
\end{equation}
We claim that the left-hand side of \eqref{eq:lincomb} is equal to
the left-hand side of~\eqref{1} with $l$ replaced by $l+1$. To see
this, we rewrite the left-hand side of $\la q^l\cdot\eqref{2}$ in the
form
\begin{multline} \label{eq:lin1}
\lambda q^l\sum_{k_1=0}^l\sum_{k_2=0}^lq^{\binom{k_1}2+\binom{k_2}2}\qbinom
    l{k_1}q\qbinom
    l{k_2}q\lambda^{k_1}(-\alpha)^{k_2}v_{n-k_1-k_2-1}\\
=\sum_{k_1=0}^{l+1}\sum_{k_2=0}^{l+1}
    q^{l-k_1+1+\binom{k_1}2+\binom{k_2}2}\qbinom
    l{k_1-1}q\qbinom
    l{k_2}q\lambda^{k_1}(-\alpha)^{k_2}v_{n-k_1-k_2},
\end{multline}
we rewrite the left-hand side of $-\al q^l\cdot\eqref{2}$ in the form
\begin{multline} \label{eq:lin2}
-\al q^l\sum_{k_1=0}^l\sum_{k_2=0}^lq^{\binom{k_1}2+\binom{k_2}2}\qbinom
    l{k_1}q\qbinom
    l{k_2}q\lambda^{k_1}(-\alpha)^{k_2}v_{n-k_1-k_2-1}\\
=\sum_{k_1=0}^{l+1}\sum_{k_2=0}^{l+1}
    q^{l-k_2+1+\binom{k_1}2+\binom{k_2}2}\qbinom
    l{k_1}q\qbinom
    l{k_2-1}q\lambda^{k_1}(-\alpha)^{k_2}v_{n-k_1-k_2},
\end{multline}
and we rewrite the left-hand side of $-\la\al q^{2l}\cdot\eqref{3}$ in the
form
\begin{multline} \label{eq:lin3}
-\la\al q^{2l}\sum_{k_1=0}^l\sum_{k_2=0}^lq^{\binom{k_1}2+\binom{k_2}2}\qbinom
    l{k_1}q\qbinom
    l{k_2}q\lambda^{k_1}(-\alpha)^{k_2}v_{n-k_1-k_2-2}\\
=\sum_{k_1=0}^{l+1}\sum_{k_2=0}^{l+1}
    q^{(l-k_1+1)+(l-k_2+1)+\binom{k_1}2+\binom{k_2}2}\qbinom
    l{k_1-1}q\qbinom
    l{k_2-1}q\lambda^{k_1}(-\alpha)^{k_2}v_{n-k_1-k_2}.
\end{multline}
By summing the left-hand side of \eqref{1} and the right-hand sides of
\eqref{eq:lin1}, \eqref{eq:lin2}, and \eqref{eq:lin3}, we obtain indeed
the left-hand side of \eqref{1} with $l$ replaced by $l+1$,
after little simplification.

We now turn our attention to the right-hand side of
\eqref{eq:lincomb}, that is, to
$$q^{l(n-l)}\sum_{s=0}^l\qbinom
    lsqq^{\binom{l-s+1}2}(-\alpha)^{s}(v_{n-l-s}+(\lambda-\alpha)
v_{n-l-s-1}-\lambda\alpha
    v_{n-l-s-2}).
$$
By~\eqref{1} with $l=1$ and $n$ replaced by $n-l-s$, this is equal to
$$q^{l(n-l)}\sum_{s=0}^l\qbinom
    lsqq^{\binom{l-s+1}2}(-\alpha)^{s}q^{n-l-s-1}(qv_{n-l-s-1}-\alpha
v_{n-l-s-2}).$$
It is not difficult to transform this
into the right-hand side of~\eqref{1} with $l$ replaced by $l+1$.
\end{proof}

As a corollary, we get for $l\ge0$ and $n\ge2l-1$
\begin{equation}
\ord_q\cD_lv_n=l(n-l).
\label{e14}
\end{equation}
Moreover, we have
\begin{equation*}
q^{-l(n-l)}\cD_lv_n\big|_{q=0}=(-\alpha)^lv_{n-2l}\big|_{q=0}.
\end{equation*}

\begin{remark}
\label{IRR1}
In the proof of Proposition~\ref{IRP2} below, under the
hypothesis $\lambda=0$, we require the estimate
$$
\ord_q\cD_lv_n>l(n-l) \quad\text{if $n<2l-1$},
$$
which complements \eqref{e14} and also follows from Lemma~\ref{IR1}.
This estimate is valid for \emph{negative} indices~$n$ as well:
recall that the definition~\eqref{e11} of our sequence $(v_n)_{n\in\mathbb Z}$
in the case $\lambda=0$ and $\alpha\ne0$ results in
$$
v_{-n-1}=q^n(v_{-n}+\alpha^{-n}),
$$
whence $\ord_qv_{-n}=n-1$ for $n\ge1$, implying the desired estimate
for negative indices.
\end{remark}

\begin{proposition}
\label{IRP1}
Let $\lambda\ne0$, and let the sequence $v_n$ be given by~\eqref{e11}.
Then the Hankel determinant $V_n=\det_{0\le i,j\le n-1}(v_{i+j})$, viewed
as an analytic function \textup(in fact, a polynomial\textup) in
$q$, $\alpha$, $\lambda$, and~$\mu$, admits the representation
\begin{equation*}
V_n=\begin{cases}
\alpha^{n(n-1)/2}\lambda^{n(n-2)/4}(\lambda-(\lambda+\alpha)\mu)^{n/2}
\cdot q^{e_0(n)}+O(q^{e_0(n)+1}) &\kern-1.3cm\text{if $n$ is even}, \\
\alpha^{n(n-1)/2}\lambda^{(n-1)^2/4}(\mu-1)(\lambda-(\lambda+\alpha)\mu)^{(n-1)/2}
\cdot q^{e_0(n)}+O(q^{e_0(n)+1})\\ &\kern-1.3cm\text{if $n$ is odd},
\end{cases}
\end{equation*}
where
\begin{equation}
e_0(n)=\frac{n(n-1)(n-2)}6=\binom n3.
\label{e15}
\end{equation}
In particular, its $q$-order under any specialization of $\alpha$, $\lambda$, and~$\mu$
is at least~$e_0(n)$.
\end{proposition}

\begin{proof}
We act on the $i$-th row of the matrix
$(v_{i+j})_{0\le i,j\le n-1}$ by the operator $\cD_{\fl{i/2}}$;
doing this for $i=n-1,n-2,\dots,1,0$ (in this order!).
By definition~\eqref{e12} we have a
sequence of elementary row operations, hence the new matrix
with entries $a_{ij}=\cD_{\fl{i/2}}v_{i+j}$, $0\le i,j\le n-1$,
has the same determinant~$V_n$. According to~\eqref{e13},
$$
e_{ij}:=\ord_qa_{ij}=\fl{\frac i2}\left(i+j-\fl{\frac i2}\right)
=\fl{\frac i2}\left(\cl{\frac i2}+j\right),
$$
and for a permutation $\tau$ of $\{0,1,\dots,n-1\}$ we have
$$
\sum_{i=0}^{n-1}e_{i,\tau(i)}
=\sum_{i=0}^{n-1}\fl{\frac i2}\cl{\frac i2}
+\sum_{i=0}^{n-1}\fl{\frac i2}\tau(i).
$$
We claim that
the minimal value of the latter expression is equal to $n(n-1)(n-2)/6$
and is attained, e.g., for $\tau(0)>\tau(1)>\dots>\tau(n-1)$.
To see this, first observe that,
if $\fl{i_1/2}>\fl{i_2/2}$ and $j_1>j_2$, then
$$
\fl{\frac{i_1}2}j_1+\fl{\frac{i_2}2}j_2
>\fl{\frac{i_1}2}j_2+\fl{\frac{i_2}2}j_1.
$$
Hence we necessarily have $\tau(0)>\tau(i)$
and $\tau(1)>\tau(i)$ for all $i\ge2$.
In other words, the 2-element
set $\{\tau(0),\tau(1)\}$ is $\{n-1,n-2\}$. Continuing in this
manner, we obtain $\{\tau(2),\tau(3)\}=\{n-3,n-4\}$,
$\{\tau(4),\tau(5)\}=\{n-5,n-6\}$, and so on.
It follows that indeed, for any permutation $\tau$, we have
$$
\sum_{i=0}^{n-1}e_{i,\tau(i)}\ge
\sum_{i=0}^{n-1}\fl{\frac i2}\cl{\frac i2}
+\sum_{i=0}^{\lfloor n/2\rfloor-1}i\big(n-1-2i+n-2-2i\big)=\binom n3.
$$
Moreover, the coefficient of the minimal power $q^{n(n-1)(n-2)/6}$
is equal to the determinant of the `anti-diagonal' matrix
\begin{equation}
\begin{pmatrix}
\text{\LARGE0}\!\!\!\!\! & & &
\phantom{|}\fbox{\footnotesize$\begin{matrix} v_{n-2} & v_{n-1} \\ v_{n-1} & v_n \end{matrix}$}\;
\\
& & \!\!\!\!\!\Ddots\!\!\!\!\! &
\\
& \phantom{|}\fbox{\footnotesize$\begin{matrix} (-\alpha)^lv_{n-2l-2} & (-\alpha)^lv_{n-2l-1} \\
(-\alpha)^lv_{n-2l-1} & (-\alpha)^lv_{n-2l} \end{matrix}$} & &
\\
\Ddots\!\!\!\!\! & & & \text{\LARGE0}
\end{pmatrix}
\label{e16}
\end{equation}
evaluated at $q=0$.
Here, if $n$~is odd, the left lower angle of the matrix contains just the
$1\times1$-matrix $(-\alpha)^{(n-1)/2}v_0$.
Let us compute the determinant of a $2\times2$-box in~\eqref{e16} assuming
$q=0$ throughout.
For $n\ge0$, the explicit expression \eqref{e09} for $v_n$ with
$b_j(q)=q^j-\la=-\la$ yields
$$v_n=\mu(-\la)^n-(-\la)^n\frac {1-(-\frac {\alpha} {\la})^{n+1}} {1+\frac
{\alpha} {\la}}.$$
Hence,
\begin{equation*}
\det\begin{pmatrix} (-\alpha)^lv_{n-2l-2} & (-\alpha)^lv_{n-2l-1} \\
(-\alpha)^lv_{n-2l-1} & (-\alpha)^lv_{n-2l} \end{pmatrix}
=\alpha^{n-1}(-\la)^{n-2l-1}\left(\mu\left(1+\frac {\alpha}
{\la}\right)-1\right).
\end{equation*}
Therefore the desired coefficient of $q^{n(n-1)(n-2)/6}$ in $V_n$ is equal to
$$
\prod_{l=0}^{n/2-1}\bigl(\alpha^{n-1}(-\lambda)^{n-2l-2}(\lambda-(\lambda+\alpha)\mu)\bigr)
$$
if $n$ is even, and it is equal to
\begin{equation*}
(-\alpha)^{(n-1)/2}(\mu-1)\prod_{l=0}^{(n-1)/2-1}\bigl(\alpha^{n-1}(-\lambda)^{n-2l-2}(\lambda-(\lambda+\alpha)\mu)\bigr)
\end{equation*}
if $n$ is odd.
\end{proof}

\begin{proposition}
\label{IRP2}
Let $\lambda=0$, and let the sequence $v_n$ be given
by~\eqref{e11}.
Then the Hankel determinant $V_n=\det_{0\le i,j\le n-1}(v_{i+j})$, viewed
as an analytic function \textup(in fact, a polynomial\textup) in $q$, $\alpha$, and~$\mu$,
admits the representation
\begin{alignat}{2}
\label{e17}
V_n&=(-1)^{n(n+2)/8}\alpha^{n(5n-2)/8}\mu^{n/2}
\cdot q^{e_0(n)}+O(q^{e_0(n)+1})
&\quad &\text{if $n$ is even},
\\
\label{e18}
V_n
&=(-1)^{(n-1)(n-3)/8}\alpha^{(n-1)(5n+1)/8}K_{(n-1)/2}
\cdot q^{e_0(n)}+O(q^{e_0(n)+1})
&\quad &\text{if $n$ is odd},
\end{alignat}
where the sequence $K_n=K_n(\alpha,\mu)$ is defined in~\eqref{e21}
below, and
\begin{equation}
e_0(n)=\begin{cases}
\dfrac{n(n-2)(5n-2)}{24} & \text{if $n$ is even}, \\[2.3mm]
\dfrac{n(n-1)(5n-7)}{24} & \text{if $n$ is odd}.
\end{cases}
\label{e19}
\end{equation}
In particular, the $q$-order of~$V_n$ under any specialization of $\alpha$ and~$\mu$
is at least~$e_0(n)$.
\end{proposition}

\begin{proof}
This time we act on the $i$-th row of the matrix
$(v_{i+j})_{0\le i,j\le n-1}$, for $i=n-\nobreak1,\allowbreak n-2,\dots,1,0$,
by the operator $\cD_{l_i}$, where $l_i=\min\{i,\fl{n/2}\}$. Again,
these are elementary row transformations because $\cD_l=(\alpha\cN;q)_l$
in~\eqref{e12} for $\lambda=0$.
For the entries $a_{ij}=\cD_{l_i}v_{i+j}$ of the resulting
matrix, whose determinant is~$V_n$, we have
$e_{ij}:=\ord_qa_{ij}\ge l_i(i+j-l_i)$, with equality occurring
when $i+j\ge2l_i-1$ (cf.\ Remark~\ref{IRR1}).
This fact and the fact that the sequence $(l_i)$ is non-decreasing imply,
for any permutation $\tau$ of $\{0,1,\dots,n-1\}$, that
\begin{equation}
\label{e20}
\sum_{i=0}^{n-1}e_{i,\tau(i)}
\ge\sum_{i=0}^{n-1}l_i(i+\tau(i)-l_i)
\ge\sum_{i=0}^{n-1}l_i(n-1-l_i)
=e_0(n),
\end{equation}
for $e_0(n)$ defined in~\eqref{e19}.
To get equality in~\eqref{e20}, the following two conditions should be satisfied:
(a)~for each~$i$ we have $i+\tau(i)\ge2l_i-1$, implying
$\tau(i)\ge2\fl{n/2}-i-1$ for $i\ge\fl{n/2}$, and
(b)~for $i<\fl{n/2}$ we have $\tau(i)=n-i-1$.

In the case of even~$n$, condition~(a) gives us $\tau(i)\ge n-i-1$ for $i\ge n/2$,
which in view of condition~(b) is possible if and only if $\tau(i)=n-i-1$ for
each $i=0,1,\dots,\allowbreak n-\nobreak1$. Therefore, the unique anti-diagonal product
$(-1)^{n(n-1)/2}\*\prod_{i=0}^{n-1}a_{i,n-1-i}$
provides the lowest power $q^{n(n-2)(5n-2)/24}$
in the determinant $\det(a_{ij})_{0\le i,j\le n-1}$, implying~\eqref{e17}.

If $n$~is odd, conditions (a) and (b) take the form
$$
\begin{aligned}
\tau(i)&=n-i-1 &\quad&\text{for $i=0,1,\dots,\tfrac{n-3}2$},
\\
\tau(i)&\ge n-i-2 &\quad&\text{for $i=\tfrac{n-1}2,\dots,n-1$}.
\end{aligned}
$$
In this case the coefficient of the lowest power $q^{n(n-1)(5n-7)/24}$ in~$V_n$
is equal to the determinant of the matrix
\begin{equation*}
\begin{pmatrix}
\text{\LARGE0}\!\!\!\!\! &
\phantom{|}\hbox{\footnotesize$\begin{matrix} & & & \!\!\!\!\!\!\!\!\!\! (-\alpha)^0v_{n-1} \\
& & \!\!\!\!\!\!\!\!\!\! (-\alpha)^1v_{n-3} & \\
& \!\!\!\!\!\!\!\!\!\! \Ddots & & \\
\!\!\!\!\!\!\!\!\!\! (-\alpha)^{(n-3)/2}v_2 & & & \end{matrix}$}\;
\\
\phantom{|}\fbox{\footnotesize$\begin{matrix} 0 & & \!\!\!\!\! (-\alpha)^{\fl{n/2}}v_{-1} & \!\!\!\!\! (-\alpha)^{\fl{n/2}}v_0 \\
& \Ddots & \!\!\!\!\! (-\alpha)^{\fl{n/2}}v_0 & \!\!\!\!\! (-\alpha)^{\fl{n/2}}v_1 \\
(-\alpha)^{\fl{n/2}}v_{-1} \!\!\! & \Ddots & & \vdots \\
(-\alpha)^{\fl{n/2}}v_0 \!\!\! & \dots & \dots & \!\!\!\!\! (-\alpha)^{\fl{n/2}}v_{\fl{n/2}} \end{matrix}$}
& \text{\LARGE0}
\end{pmatrix}
\end{equation*}
evaluated at $q=0$. It is clear that the non-vanishing of the coefficient
will follow from
the non-vanishing of the determinant
\begin{align*}
\left.\det\begin{pmatrix} 0 & & \!\!\! v_{-1} & \!\!\! v_0 \\
& \Ddots & \!\!\! v_0 & \!\!\! v_1 \\
v_{-1} \!\! & \Ddots & & \vdots \\
v_0 \!\! & \dots & \dots & \!\!\! v_{(n-1)/2} \end{pmatrix}\right|_{q=0}
&=\det\begin{pmatrix} 0 & & \!\!\! \mu & \!\!\! \mu-1 \\
& \Ddots & \!\!\! \mu-1 & \!\!\! -\alpha \\
\mu \!\! & \Ddots & & \vdots \\
\mu-1 \!\! & \dots & \dots & \!\!\! -\alpha^{(n-1)/2} \end{pmatrix}
\\
&=(-1)^{(n-1)(n-3)/8}K_{(n-1)/2},
\end{align*}
where
\begin{equation}
K_n=K_n(\alpha,\mu)=\det\begin{pmatrix}
\mu-1 & -\alpha & -\alpha^2 & \hdots & -\alpha^n \\
\mu & \mu-1 & -\alpha & \hdots & -\alpha^{n-1} \\
& \mu & \mu-1 & \hdots & -\alpha^{n-2} \\
& & \ddots & \ddots & \vdots \\
\!\!\!\text{\LARGE0}\!\!\! & & & \mu & \mu-1
\end{pmatrix}.
\label{e21}
\end{equation}
Summarizing we have~\eqref{e18}, and the proposition follows.
\end{proof}

\begin{remark}
Generically, the power $e_0(n)$ in \eqref{e19} is exact.
Indeed, the determinant $K_n$ is not identically zero, since
it is trivially non-zero for $\alpha=0$ and $\mu\ne1$.
In fact, we can also
write down explicit formulas for~$K_n$, since the sequence satisfies the
linear recurrence
\begin{equation*}
\begin{gathered}
K_{n+2}=(\mu-1-\alpha\mu)K_{n+1}+\alpha\mu^2K_n
\quad\text{for}\; n=0,1,2,\dots,
\\
K_0=\mu-1, \quad K_1=(\mu-1)^2+\alpha\mu.
\end{gathered}
\end{equation*}
\end{remark}

To find an (asymptotic) upper bound for
our Hankel determinant
$$\det_{0\le i,j\le n-1}(v_{i+j}),$$ where
\begin{equation}
v_n=\sum_{k=n+1}^\infty\frac{\alpha^k}{\prod_{j=n+1}^k(q^j-\lambda)}
\quad\text{for}\; n=0,1,2,\dots,
\label{e23}
\end{equation}
we will use the difference operator
\begin{equation}
\tcD_l=(\alpha q^{-1}\cN;q^{-1})_l=\prod_{k=1}^l(\mathcal I-\alpha q^{-k}\cN).
\label{e24}
\end{equation}

\begin{remark}
\label{IRR2}
The operators~\eqref{e12} and \eqref{e24} are directly related by
$$
\cD_lv_n=q^{ln-\binom l2}\tcD_lv_{n-l},
$$
where $l\ge0$ and $n\in\mathbb Z$. This is seen by applying the
$q$-binomial theorem \eqref{eq:qbin} to \eqref{e24}, and by comparing
the result with \eqref{e13}.
Equivalently,
$$
\tcD_lv_n=q^{\binom l2-l(n+l)}\cD_lv_{n+l}.
$$
\end{remark}

\begin{lemma} \label{lem:<1}
Let $q$, $\al$, $\la$ be complex numbers with $\vert q\vert>1$, $\al\ne0$,
and $\la$ not a {\em(}positive{\em)} power of $q$. Then,
for $0\le l\le n$, we have
$$
\vert\tcD_lv_n\vert\le \begin{cases}
\vert q\vert^{  \binom l2 - n l}C_1^{n+1}&\text{if }\la\ne0,\\
\vert q\vert^{  - n l}C_2^{n+1}&\text{if }\la=0,
\end{cases}
$$
where $C_1$ and $C_2$ are positive real numbers not depending on $n$ and~$l$.
\end{lemma}

\begin{proof}
Let first $\la\ne0$. From \eqref{e23}, it is easy to see that we can find a
real number $C>1$ (depending on $q$, $\al$ and $\la$, but not on
$n$), such that $\vert
v_n\vert\le C^{n+1}$. Making use of the $q$-binomial theorem \eqref{eq:qbin},
of the fact that
$$\left\vert\begin{bmatrix}
m\\k\end{bmatrix}_q\right\vert\le
\begin{bmatrix}
m\\k\end{bmatrix}_{\vert q\vert} $$
(following from the polynomiality in $q$ of $\left[\begin{smallmatrix}
m\\k\end{smallmatrix}\right]_q$, the coefficients being non-negative
integers), and of our observation in Remark~\ref{IRR2}, we have
\begin{align*}
\vert\tcD_lv_n\vert
&= \vert q\vert^{\binom l2-l(n+l)}\vert \cD_lv_{n+l}\vert\\
&\le \vert q\vert^{-\binom {l+1}2-ln}
\sum _{k_1=0} ^{l}\vert q\vert^{\binom {k_1}2}\begin{bmatrix}
l\\k_1\end{bmatrix}_{\vert q\vert} \vert\la\vert^{k_1}
\sum _{k_2=0} ^{l}
\vert q\vert^{\binom {k_2}2}\begin{bmatrix}
l\\k_2\end{bmatrix}_{\vert q\vert} \vert\al\vert^{k_2} \vert v_{n+l-k_1-k_2}\vert\\
&\le \vert q\vert^{-\binom {l+1}2-ln}
C^{n+l+1}\prod _{j=0} ^{l-1}(1+\vert\la\vert\,\vert q\vert^j)
(1+\vert\alpha\vert\,\vert q\vert^j)\\
&\le \vert q\vert^{\binom l2-l-ln}
C^{n+l+1}(1+\vert\la\vert)^l(1+\vert\al\vert)^l,
\end{align*}
which, in view of $l\le n$, is exactly in line with the first assertion of the lemma.

In the case $\la=0$, we can proceed in the same way. The only
difference is that the above sum over $k_1$ reduces to just the summand
for $k_1=0$. Hence, we obtain
\begin{align*}
\vert\tcD_lv_n\vert
&\le \vert q\vert^{-\binom {l+1}2-ln}
\sum _{k_2=0} ^{l}
\vert q\vert^{\binom {k_2}2}\begin{bmatrix}
l\\k_2\end{bmatrix}_{\vert q\vert} \vert\al\vert^{k_2} \vert v_{n+l-k_2}\vert\\
&\le \vert q\vert^{-\binom {l+1}2-ln}
C^{n+l+1}\prod _{j=0} ^{l-1}
(1+\vert\alpha\vert\,\vert q\vert^j)\\
&\le \vert q\vert^{-l-ln}
C^{n+l+1}(1+\vert\al\vert)^l,
\end{align*}
which is exactly in line with the second assertion of the lemma.
\end{proof}

\begin{proposition}
\label{IRP3}
Let the sequence $v_n$ be given by~\eqref{e23}.
Then, as $n$ tends to $\infty$,
the Hankel determinant $V_n=\det_{0\le i,j\le n-1}(v_{i+j})$
is asymptotically
\begin{equation}
\vert V_n\vert\le\begin{cases}
\vert q\vert^{-n^3/3}\exp(O(n^2))
&\text{if $\lambda\ne0$},
\\
\vert q\vert^{-n^3/2}\exp(O(n^2))
&\text{if $\lambda=0$}.
\end{cases}
\label{e26}
\end{equation}
\end{proposition}

\begin{proof}
Acting by the operator $\tcD_i$ on the
$i$-th row of the matrix $(v_{i+j})_{0\le i,j\le n-1}$
(this results in elementary row operations according to~\eqref{e24})
we get the matrix $(a_{ij})_{0\le i,j\le n-1}$ with entries
$
a_{ij}=\tcD_iv_{i+j},$
whose determinant is equal to $V_n$.

Let now $\la\ne0$.
Writing $\mathfrak S_n$ for the
symmetric group on $\{0,1,2,\dots,n-1\}$, we have
\begin{align*}
\vert V_n\vert&\le n!\max_{\tau\in \mathfrak S_n}
\prod _{i=0} ^{n-1}\vert  \tcD_iv_{i+\tau(i)}\vert\\
&\le n!\max_{\tau\in \mathfrak S_n}\prod _{i=0} ^{n-1}
\vert q\vert^{  \binom i2 -(i+\tau(i)) i}
C_1^{i+\tau(i)+1}\\
&\le \exp(O(n^2))
\prod _{i=0} ^{n-1}
{\vert q\vert^{  \binom i2 -(n-1) i}},
\end{align*}
where, to go from the next-to-last to the last line, we used again the
fact that the permutation achieving the maximum is the permutation
sending $i$ to $n-i-1$, $i=0,1,\dots,n-1$.
This implies the first claim of the proposition.

On the other hand, if $\la=0$, then we have
\begin{align*}
\vert V_n\vert&\le n!\max_{\tau\in \mathfrak S_n}
\prod _{i=0} ^{n-1}\vert  \tcD_iv_{i+\tau(i)}\vert\\
&\le n!\max_{\tau\in \mathfrak S_n}\prod _{i=0} ^{n-1}
\vert q\vert^{  - (i+\tau(i)) i}C_2^{i+\tau(i)+1}
\\
&\le \exp(O(n^2))
\prod _{i=0} ^{n-1}
{\vert q\vert^{  -(n-1) i}},
\end{align*}
implying the second claim of the proposition.
\end{proof}

\begin{remark}
\label{IRR3}
Using an analytic method for the (entire) generating series
$\sum_{n=0}^\infty v_nz^n$, Choulet \cite[Lemmas~3.3 and~3.4]{Ch}
proves estimates that may be informally summarized
in our settings as follows:
\begin{equation}
\ord_qV_n\ge\begin{cases}
\frac5{24}n^3+O(n^2) &\text{if $\lambda=0$}, \\
\frac16n^3+O(n^2) &\text{if $\lambda\ne0$},
\end{cases}
\qquad\text{as}\quad n\to\infty,
\label{e28}
\end{equation}
and
\begin{equation}
\vert V_n\vert\le\begin{cases}
\vert q\vert^{-\frac12n^3+O(n^2)} &\text{if $\lambda=0$}, \\
\vert q\vert^{-\frac13n^3+O(n^2)} &\text{if $\lambda\ne0$},
\end{cases}
\qquad\text{as}\quad n\to\infty.
\label{e29}
\end{equation}
Therefore, our Propositions~\ref{IRP1}--\ref{IRP2}
sharpen the estimate \eqref{e28} of Choulet, while our
Proposition~\ref{IRP3} provides a different proof of \eqref{e29}.
(Strictly speaking, Choulet did not arrive at the better estimate in~\eqref{e29}
for the case $\lambda=0$ himself; this had been done earlier by B\'ezivin~\cite{Be}
using elementary considerations.)
On the other hand, it is Choulet's method that suggested to us the form of
the difference operators~\eqref{e12} and~\eqref{e24}.
\end{remark}

\section{Cyclotomic factorization}
\label{s4}

We now turn to the general Hankel determinant~\eqref{e07} with the
sequence $(v_n)$ defined in \eqref{e08} and \eqref{e09}.
Let us fix the notation
$$
\Phi_l(q)=\underset{\gcd(j,l)=1}{\prod_{1\le j\le l}}(q-e^{2\pi j\sqrt{-1}/l}),
\qquad l=1,2,\dots,
$$
for the cyclotomic polynomials.

\begin{proposition}
\label{l0}
For any integer $l$ in the range $1\le l<n/2$, the Hankel determinant
$V_n=\det_{0\le i,j\le n-1}(v_{i+j})$, where the $v_n$'s are
given in~\eqref{e09} with $b_j(q)=b(q^j)$ for a polynomial $b(\,\cdot\,)$
of arbitrary degree, is divisible by $\Phi_l(q)^{e_l(n)}$, where
\begin{equation} \label{e35}
e_l(n)=\sum_{i=0}^{n-1}\left(\left\lfloor\frac {i+l}{3l}\right\rfloor+\left\lfloor
\frac {i} {3l}\right\rfloor\right).
\end{equation}
\end{proposition}

\begin{remark}
{}From \eqref{e35}, it is straightforward to compute a compact
formula for $e_l(n)$, namely,
\begin{equation*}
e_l(n)=\frac{(n-j)(n+j-2l)}{3l}+\begin{cases}
0 & \text{if $0\le j<2l$}, \\
j-2l & \text{if $2l\le j<3l$},
\end{cases}
\quad\text{where}\
n\equiv j\pmod{3l}.
\end{equation*}
(In other words, $j/(3l)$ is the fractional part of $n/(3l)$.)
In particular,
\begin{equation*}
e_1(n)=\biggl\lfloor\frac{(n-1)^2}3\biggr\rfloor
=\begin{cases}
\dfrac{(n-1)^2}3 & \text{if $n\equiv1\pmod3$}, \\[2.3mm]
\dfrac{n(n-2)}3 & \text{otherwise},
\end{cases}
\qquad
e_2(n)=\biggl\lfloor\frac{(n-2)^2}6\biggr\rfloor.
\end{equation*}
\end{remark}

\begin{proof}
We have to do some preparatory work first, before we are in the
position to embark on the ``actual" proof of the divisibility
assertion in the proposition. The central part of this preparatory
work is the identity \eqref{30}.

Changing notation slightly, recall that, for $n\ge1$, we have
$$v_n=b(q^n)v_{n-1}-\alpha^n,$$
with $b(x)=x-\la$. (In fact, the subsequent arguments hold for {\it
any} function $b(x)$. It is therefore that we shall write $b(\,.\,)$
in the sequel instead of its explicit form which is of relevance in our
context.) Hence, for $n\ge l\ge1$, we have
\begin{equation}\label{22}
    v_n=v_{n-l}\prod_{k=0}^{l-1}b(q^{n-k})-\sum_{j=0}^{l-1}\alpha^{n-j}\prod_{k=0}^{j-1}b(q^{n-k}).
\end{equation}
Writing
\begin{equation*}
    P_j(n,q)=\prod_{k=0}^{j-1}b(q^{n-k})
\end{equation*}
for non-negative integers $j$, the recurrence \eqref{22} takes the form
\begin{equation}\label{24}
    v_n=v_{n-l}P_l(n,q)-\sum_{j=0}^{l-1}\alpha^{n-j}P_j(n,q).
\end{equation}

Fix a positive integer $l$ and a primitive $l$-th root of unity $\zeta$.
For integers $j,t,n$ with $j,t\ge0$, set
\begin{equation}\label{25}
    P_j^{(t)}(n)=\frac{d^t}{dq^t}P_j(n,q)\Bigr|_{q=\zeta}.
\end{equation}
In particular, we have
\begin{equation}\label{26}
    P_0^{(t)}(n)=\begin{cases}1 & \text{if $t=0$}, \\
                     0 & \text{if $t\ge1$}.
                 \end{cases}
\end{equation}
For $j\ge1$, we have
\begin{equation}\label{27}
    \frac{d^t}{dq^t}P_j(n,q)=\sum_{t_0+\dots+t_{j-1}=t}\frac{t!}{t_0!\dotsb
    t_{j-1}!}\prod_{k=0}^{j-1}\frac{d^{t_k}}{dq^{t_k}}b(q^{n-k}).
\end{equation}


Applying the Fa\`a di Bruno formula (cf.\ \cite[Sec.~3.4]{ComtAA};
but see also \cite{CraiAA,JohWAE}) we get
\begin{equation}\label{28}
    \frac{d^{t}}{dq^{t}}b(q^{n-k})
    =\sum_{m_1+2m_2+\dots+tm_t=t}\frac{t!}{m_1!\dotsb m_t!}b^{(m)}(q^{n-k})
    \prod_{\nu=1}^t\left(\binom{n-k}\nu q^{n-k-\nu}\right)^{m_\nu},
\end{equation}
where $m=m_1+\dots+m_k$.

It is straightforward to see that Equations
\eqref{25}--\eqref{28} imply that, for any non-negative integers $j$ and $t$,
the quantity $P_j^{(t)}(n)$, as a function in $n$, is an
$l$-quasi-polynomial of degree at most $t$, where
an $l$-\emph{quasi-polynomial of degree at most\/} $t$ is
a sequence of complex numbers $(Q(n))_{n\in\ZZ}$ of the form
\begin{equation*}
    Q(n)=\sum_{\nu=0}^ta_\nu(n)n^\nu,
\end{equation*}
the sequence $\{a_\nu(n)\}_{n\in\ZZ}$ being
$l$-periodic for each~$\nu$ (cf.\ \cite[Sec.~4.4]{StanAP}).
We denote the set of all
$l$-quasi-polynomials of degree at most~$t$ by $\mathcal Q_t=\mathcal Q_t(l)$.
Furthermore note that if $Q(n)\in\mathcal Q_t$, with $t>0$, then
$Q(n)-Q(n-l)\in\mathcal Q_{t-1}$.
These facts will be used repeatedly.

Our next observation is that the sequence $(P_l^{(0)}(n))_{n\in\ZZ}$
is constant, where\break
$    P_l^{(0)}(n)=P_l^{(0)}(0)=\prod_{k=0}^{l-1}b(\zeta^{k})$.
We let
\begin{equation*}
    B=B_l=\prod_{k=0}^{l-1}b(\zeta^{k}).
\end{equation*}
Clearly, $B$~is independent of the particular choice of the
primitive $l$-th root of unity $\zeta$.

Now we introduce the difference operators
\begin{equation*}
    \mathcal F=\mathcal I-B\mathcal N^l
\end{equation*}
and
\begin{gather*}
    \mathcal G=\mathcal I-\alpha^l\mathcal N^l,
\end{gather*}
where $\mathcal I$ and $\mathcal N$ have the same meaning as earlier.
Clearly, the operators $\mathcal F$ and $\mathcal G$ commute.

It is easy to check, using elementary facts of difference calculus that,
for any $Q(n)\in\mathcal Q_t$, we have
\begin{equation} \label{eq:G}
    \mathcal G^{t+1}\bigl(Q(n)\alpha^n\bigr)=0.
\end{equation}

For non-negative integers $m$ and $n$, let
\begin{equation*}
    v_n^{(m)}=\frac{\partial^m v_n}{\partial
    q^m}\Bigr|_{q=\zeta}.
\end{equation*}
By differentiating both sides of \eqref{24} $m$~times, and
by subsequently substituting $q=\zeta$,
we obtain for $n\ge l\ge1$ the equation
\begin{equation}\label{29}
    v_n^{(m)}=\sum_{\nu=0}^m\binom m\nu
    P_l^{(m-\nu)}(n)v_{n-l}^{(\nu)}-\sum_{j=0}^{l-1}P_j^{(m)}(n)\alpha^{n-j}.
\end{equation}

We now claim that for an arbitrary $Q(n)\in\mathcal Q_t$, for
non-negative integers $t$ and $m$, and for any integer $n\ge(2t+3m+2)l$, we have
\begin{equation}\label{30}
    \mathcal F^{t+2m+1}\mathcal G^{t+m+1}\bigl(Q(n)v_n^{(m)}\bigr)=0.
\end{equation}
We prove this claim by a double induction: the external induction is over $m$, while
the inner induction is over~$t$.

We start by proving \eqref{30} for $m=0$, by doing an induction
over $t$.
We put $m=0$ in \eqref{29}, and rewrite the resulting equation in the form
\begin{equation*}
    \mathcal F(v_n^{(0)})=-\sum_{j=0}^{l-1}P_j^{(0)}(n)\alpha^{n-j}.
\end{equation*}
We apply the operator $\mathcal G$ on both sides. By \eqref{eq:G},
this implies
\begin{equation*}
    \mathcal F\mathcal G(v_n^{(0)})=0,
\end{equation*}
which in turn implies
\begin{equation*}
    \mathcal F\mathcal G(Q(n)v_n^{(0)})=0
\end{equation*}
for any $l$-periodic function $Q(n)$, since the operators $\mathcal F$
and $\mathcal G$ are both polynomials in $\mathcal N^l$.
This is exactly \eqref{30} for $m=t=0$.

Now let us assume that \eqref{30} is proved for $m=0$ and for $t-1$
instead of $t$.
If $Q(n)\in\mathcal Q_t$ with $t>0$, then \eqref{29} implies
\begin{equation*}
    \mathcal
    F\bigl(Q(n)v_n^{(0)}\bigr)=B(Q(n)-Q(n-l))v_{n-l}^{(0)}-\sum_{j=0}^{l-1}Q(n)P_j^{(0)}(n)\alpha^{n-j}.
\end{equation*}
After application of $\mathcal F^t\mathcal G^{t+1}$ on both sides, we obtain
\begin{multline*}
    \mathcal
    F^{t+1}\mathcal G^{t+1}\bigl(Q(n)v_n^{(0)}\bigr)
=B\mathcal G\mathcal F^t\mathcal G^{t}\bigl((Q(n)-Q(n-l))v_{n-l}^{(0)}\bigr)\\
-\sum_{j=0}^{l-1}\mathcal F^t\mathcal G^{t+1}\bigl(Q(n)P_j^{(0)}(n)\alpha^{n-j}\bigr).
\end{multline*}
The summands in the sum over $j$ vanish because of \eqref{eq:G}, while
the first expression on the right-hand side vanishes because of the
induction hypothesis. (Recall that $Q(n)-Q(n-l)\in \mathcal Q_{t-1}$.)
This proves \eqref{30} for $m=0$ and
arbitrary $t$.

Now we assume that \eqref{30} is proved for $0,1,\dots,m-1$
instead of $m$ and arbitrary $t$.
For $m>0$, we write~\eqref{29} as
\begin{equation*}
    \mathcal Fv_n^{(m)}=\sum_{\nu=0}^{m-1}\binom m\nu
    P_l^{(m-\nu)}(n)v_{n-l}^{(\nu)}-\sum_{j=0}^{l-1}P_j^{(m)}(n)\alpha^{n-j}.
\end{equation*}
We apply $\mathcal F^{2m}\mathcal G^{m+1}$ on both sides, to obtain
\begin{multline*}
    \mathcal F^{2m+1}\mathcal G^{m+1}(v_n^{(m)})
=\sum_{\nu=0}^{m-1}\binom m\nu
   \mathcal F^{m-\nu-1}
\mathcal F^{(m-\nu)+2\nu+1}\mathcal G^{(m-\nu)+\nu+1}\bigl(
P_l^{(m-\nu)}(n)v_{n-l}^{(\nu)}\bigr)\\
-\sum_{j=0}^{l-1}\mathcal F^{2m}\mathcal
G^{m+1}\bigl(P_j^{(m)}(n)\alpha^{n-j}\bigr).
\end{multline*}
Again, the summands in the sum over $j$ vanish because of \eqref{eq:G}, while
the first expression on the right-hand side vanishes because of the
induction hypothesis. This establishes \eqref{30} for $m$ and
$t=0$.

In order to prove \eqref{30} for $m$ and {\it arbitrary} $t$, we do
again an induction over $t$. We already know that \eqref{30} is true
for $t=0$. Let us assume that \eqref{30} is true for $t-1$ instead of $t$.
We multiply both sides of \eqref{29} by $Q(n)$, and we apply $\mathcal
F$ on both sides. The resulting equation can then be written in the form
\begin{multline*}
    \mathcal
F\bigl(Q(n)v_n^{(m)}\bigr)=B(Q(n)-Q(n-l))v_{n-l}^{(m)}
+\sum_{\nu=0}^{m-1}\binom m\nu
    Q(n)P_l^{(m-\nu)}(n)v_{n-l}^{(\nu)}\\
    -\sum_{j=0}^{l-1}Q(n)P_j^{(m)}(n)\alpha^{n-j}.
\end{multline*}
After application of $\mathcal F^{t+2m}\mathcal G^{t+m+1}$ on
both sides, we arrive at
\begin{multline*}
    \mathcal F^{t+2m+1}\mathcal
G^{t+m+1}\bigl(Q(n)v_n^{(m)}\bigr)=
B\mathcal G\mathcal F^{t+2m}\mathcal
G^{t+m}\bigl((Q(n)-Q(n-l))v_{n-l}^{(m)}\bigr)\\
+\sum_{\nu=0}^{m-1}\binom m\nu
  \mathcal F^{m-\nu-1}\mathcal F^{ (t+m-\nu)+2\nu+1}\mathcal
G^{(t+m-\nu)+\nu+1}\bigl(
Q(n)P_l^{(m-\nu)}(n)v_{n-l}^{(\nu)}\bigr)\\
    -\sum_{j=0}^{l-1}\mathcal F^{t+2m}\mathcal
G^{t+m+1}\bigl(Q(n)P_j^{(m)}(n)\alpha^{n-j}\bigr).
\end{multline*}
Again, by \eqref{eq:G} and the induction hypothesis, the right-hand
side in this identity vanishes.
Thus, \eqref{30} is completely proved.

\medskip
We are now ready to treat the Hankel determinant $V_n$.
In fact, we need \eqref{30} only for $t=0$. (For the proof, it was
however necessary to play with $t$.)
What \eqref{30} for $t=0$ says is
that, for $m\ge1$ and $n\ge(3m-1)l$, the polynomial (in~$q$) $w_{m,n}=(\mathcal
I-B\mathcal N^l)^{2m-1}(\mathcal I-\alpha^l\mathcal N^l)^{m}v_n$
satisfies
$$\left.\frac{\partial^jw_{m,n}}{\partial q^j}\right|_{q=\zeta}=0,\qquad0\le j\le m-1,$$
and, for any choice of the primitive $l$-th root of unity~$\zeta$.
Hence, we have
\begin{equation} \label{eq:Phidiv}
\Phi_l(q)^{m}\mid w_{m,n}.
\end{equation}
This reasoning also shows that for the polynomial
$$
\wt w_{m,n}
=(\mathcal I-B\mathcal N^l)^{2m}(\mathcal I-\alpha^l\mathcal N^l)^{m}v_n
=(\mathcal I-B\mathcal N^l)w_{m,n}
=w_{m,n}-Bw_{m,n-l}
$$
we have
\begin{equation} \label{eq:Phidiv2}
\Phi_l(q)^{m}\mid\wt w_{m,n},
\end{equation}
as long as $n\ge 3ml$.

For $n-1\ge i\ge 2l$ we apply the operator
$(\mathcal I-B\mathcal N^l)^{2l_i-1}(\mathcal I-\alpha^l\mathcal
N^l)^{l_i}$, where $l_i=\lfloor(i+l)/(3l)\rfloor$ to the $i$-th row
of the Hankel matrix $(v_{i+j})_{0\le i,j\le n-1}$,
and subsequently, for $n-1\ge j\ge3l$, we apply the operator
$(\mathcal I-B\mathcal N^l)^{2m_j}(\mathcal I-\alpha^l\mathcal
N^l)^{m_j}$, where $m_j=\lfloor j/(3l)\rfloor$ to the $j$-th column.
The resulting matrix has entries
$$
\begin{cases}
\wt w_{\ep(i,j),i+j} & \text{if $i<2l$ and $j\ge3l$}, \\
w_{\ep(i,j),i+j} & \text{otherwise},
\end{cases}
$$
where
$$\ep(i,j)=l_i+m_j=\lfloor(i+l)/(3l)\rfloor+\lfloor j/(3l)\rfloor,$$
with the convention that $w_{0,n}=v_n$.
As earlier, since the above operations correspond to row and column
operations, the determinant of the resulting matrix is still equal to~$V_n$.

In view of \eqref{eq:Phidiv} and \eqref{eq:Phidiv2},
it remains to observe that, for an arbitrary permutation~$\tau$,
$$
\sum_{i=0}^{n-1}\ep(i,\tau(i))=
\sum_{i=0}^{n-1}\left(\left\lfloor\frac {i+l}{3l}\right\rfloor+\left\lfloor
\frac {\tau(i)} {3l}\right\rfloor\right)\\
=\sum_{i=0}^{n-1}\left(\left\lfloor\frac {i+l}{3l}\right\rfloor+\left\lfloor
\frac {i} {3l}\right\rfloor\right)
=e_l(n),
$$
This completes the proof of the proposition.
\end{proof}


Propositions \ref{IRP1}, \ref{IRP2} and~\ref{l0} may be summarized
as follows: The Hankel determinant~\eqref{e07} of the sequence~\eqref{e11}
admits the factorization
\begin{equation}
\label{e31}
V_n=\Delta_n(q)\cdot\wt V_n,
\end{equation}
where
\begin{equation}
\label{e32}
\Delta_n(q)
=q^{e_0(n)}\prod_{l\ge1}\Phi_l(q)^{e_l(n)}
\end{equation}
and the exponents $e_0(n),e_1(n),e_2(n),\dots$ are given by very simple formulas.
Since
$$
\frac{\log|\Phi_l(q)|}{\log|q|}=\varphi(l)+O(1)
\quad\text{as }l\to\infty
$$
and
$$
e_l(n)=O\biggl(\frac{n^2}l\biggr)\quad\text{as $n\to\infty$ uniformly in $l\ge1$},
$$
the asymptotic behaviour of $\Delta_n(q)$ as $n\to\infty$ is governed by the degree of the polynomial,
\begin{equation} \label{eq:logDelta}
\frac{\log|\Delta_n(q)|}{\log|q|}
\sim\deg\Delta_n=e_0(n)+\sum_{l=1}^\infty e_l(n)\varphi(l)
\quad\text{as}\; n\to\infty.
\end{equation}

The following lemma enables us to determine the asymptotic behaviour of the sum on
the right-hand side of \eqref{eq:logDelta} as $n\to\infty$.

\begin{lemma} \label{eq:sumel}
Let $a$ and $c$ be real numbers with $0\le c<a$.
Then, as $n\to\infty$,
$$\sum_{l\ge1}\varphi(l)\sum_{i=0}^{n}\left\lfloor\frac{i+cl}{al}
\right\rfloor=
\frac{n^3}{\pi^2}
\sum_{m=1}^{\infty}\frac1{(am-c)^2}+O(n^2\log^2n).
$$
\end{lemma}
\begin{proof}
By interchanging summations, we have
\begin{equation} \label{eq:sumsum}
\sum_{l=1}^{\infty}\varphi(l)\sum_{i=0}^{n}\left\lfloor\frac{i+cl}{al}
\right\rfloor=
\sum_{i=0}^{n}\sum_{l=1}^{\infty}\left\lfloor\frac{i+cl}{al}\right\rfloor
\varphi(l).
\end{equation}
Writing
$\Sigma(x)=\sum_{l\le x}\varphi(l)$,
the expression \eqref{eq:sumsum} can be rewritten in the form
$$
\sum_{i=0}^{n}\sum_{m=1}^{\infty}
m\left(\Sigma\left(\frac i{am-c}\right)
-\Sigma\left(\frac i{a(m+1)-c}\right)\right)
=
\sum_{i=0}^{n}\sum_{m=1}^{\infty}
\Sigma\left(\frac i{am-c}\right).
$$
Using Mertens' classical asymptotic formula
(cf.\ \cite[p.~268, Theorem~330]{HW})
$$\Sigma(x)=\frac{3x^2}{\pi^2}+O(x\log x)
\qquad \text {as }x\to\infty,$$
we obtain the following asymptotic estimate for \eqref{eq:sumsum}:
\begin{multline*}
\sum_{i=\cl{a-c}}^{ n}\sum_{m=1}^{\fl{(i+c)/a}}
\left(\frac{3i^2}{\pi^2(am-c)^2}+O\left(\frac{i\log
n}m\right)\right)+O(1)\\
=\sum_{i=\cl{a-c}}^{ n}\left(\frac{3i^2}{\pi^2}
\sum_{m=1}^{\infty}\frac1{(am-c)^2}+O(i\log^2n)\right)+O(1),
\end{multline*}
which immediately implies the assertion of the lemma.
\end{proof}

Using Formula~\eqref{e35} for $e_l(n)$ and
the asymptotics from Lemma~\ref{eq:sumel}, we obtain
\begin{equation}
\sum_{l=1}^\infty e_l(n)\varphi(l)
\approx 0.05301135n^3
\quad\text{as}\ n\to\infty,
\label{e34}
\end{equation}
where the exact value of the constant in~\eqref{e34} is
\begin{equation*}
\frac1{54}+\frac1{\pi^2}\sum_{m=1}^\infty\frac1{(3m-1)^2}
=\frac5{54}-\frac{\Im\Li_2(e^{2\pi \sqrt{-1}/3})}{\pi^2\sqrt3}.
\end{equation*}

Summarizing our findings from Propositions~\ref{IRP1} and \ref{IRP2},
as well as from \eqref{eq:logDelta} and \eqref{e34},
we have the following result for the asymptotic degree of
$\Delta_n(q)$.

\begin{proposition} \label{IRP5}
Let $\Delta_n(q)$ be defined as in \eqref{e32}.
Then, as $n\to\infty$, we have
$$
\frac{\log|\Delta_n(q)|}{\log|q|}
\sim \deg_q \Delta_n(q)\sim Bn^3,
$$
where
$$B=\frac5{54}-\frac{\Im\Li_2(e^{2\pi \sqrt{-1}/3})}{\pi^2\sqrt3}+
\begin{cases} \dfrac {5} {24}&\text {if }\la=0,\\[2mm]
\dfrac {1} {6}&\text {if }\la\ne0.
\end{cases}
$$

\end{proposition}

\begin{remark}
\label{cyclR}
If $\lambda$~is a root of unity, expected formulas for the exponents
of the cyclotomic factors of $V_n$
obey a different law. To write them down in the ($q$-exponential) case $\lambda=1$,
we represent the polynomial $\Delta_n$ from the factorization~\eqref{e31} in
the form
\begin{equation*}
\Delta_n
=q^{e_0(n)}\prod_{l\ge1}(q^l-1)^{\wt e_l(n)}.
\end{equation*}
Then
\begin{equation}
\label{e37}
\wt e_l(n)
=2\max\{0,n-2l\},
\end{equation}
which, together with~\eqref{e15}, implies that, in the case
$\lambda=1$, we have
\begin{equation*}
\deg\Delta_n=\begin{cases}
\dfrac{n(n-1)^2}4 & \text{if $n$ is odd}, \\[2.3mm]
\dfrac{n^2(n-2)}4 & \text{if $n$ is even}.
\end{cases}
\end{equation*}
It is of definite interest to prove also these formulas for the cyclotomic exponents.
\end{remark}

\section{Arithmetic ingredients}
\label{s5}

In this section we provide the proof of Theorems~\ref{t1} and
\ref{t2}. It rests on Propositions~\ref{IRP3}--\ref{IRP5},
and two additional auxiliary results, given in Lemmas~\ref{lem:Kron}
and \ref{lem:conj} below. The first one says that, under
certain arithmetic constraints on the complex parameters
$\alpha$, $\lambda$ and~$q$, if we generalize
our sequence $v_n$ to $v_n(x)$, where $v_0(x)=x-1$ and
\begin{equation}
v_n(x)=v_{n-1}(x)\cdot(q^n-\lambda)-\alpha^n
\qquad\text{for}\quad n=1,2,\dots,
\label{e52}
\end{equation}
then the corresponding Hankel determinant
\begin{equation}
V_n(x)=\det_{0\le i,j\le n-1}\bigl(v_{i+j}(x)\bigr)
\label{e51}
\end{equation}
is non-zero infinitely often, while the second
establishes a (crude) asymptotic upper bound for it.
The reader should note that $v_n(x)$ becomes our previous $v_n$ defined in
\eqref{e11} if $x=\mu$, where $\mu$ is given by \eqref{eq:mu}.
Hence, if $x=\mu$, the Hankel determinant $V_n(x)$ becomes
our earlier Hankel determinant $V_n$.

\begin{lemma} \label{lem:Kron}
Let $\al,\la,q,x$ be complex numbers with $\al\ne0$,
$\la\notin q^{\mathbb Z_{>0}}$, and
$\alpha\notin-\lambda q^{\mathbb Z_{>0}}$.
Then there are infinitely many positive integers
$n$ such that $V_n(x)\ne0$.
\end{lemma}
\begin{proof}
Writing the relation~\eqref{e52} for the generating series
$$
G(z)=G_x(z)=\sum_{n=0}^\infty v_n(x)z^n
$$
we arrive at
\begin{align*}
G(z)
&=v_0(x)+z\bigl(qG(qz)-\lambda G(z)\bigr)-\sum_{n=1}^\infty\alpha^nz^n
\\
&=z\bigl(qG(qz)-\lambda G(z)\bigr)+x-\frac1{1-\alpha z}.
\end{align*}
Equivalently,
\begin{equation}
(1+\lambda z)G(z)-qzG(qz)
=x-\frac1{1-\alpha z}.
\label{e53}
\end{equation}
We claim that this equation does not have a rational function solution
unless $\alpha\in-\lambda q^{\mathbb Z_{>0}}$. Indeed, if
$z=1/\beta$, $z=1/(q\beta)$, \dots, $z=1/(q^{k-1}\beta)$ are
poles of $G(z)$, for some~$k$, then
$z=1/(q\beta)$, $z=1/(q^2\beta)$, \dots, $z=1/(q^{k}\beta)$ are
poles of $G(qz)$. Hence,
the only way that this is possible in \eqref{e53} is that the factor
$1+\la z$ cancels the pole $z=1/\beta$ of $G(z)$, while the term $-1/(1-\al z)$
on the right-hand side cancels the pole $z=1/(q^k\beta)$ of~$G(qz)$.

By a result of Kronecker (see \cite[pp.~566--567]{Kro} or
\cite[Division~7, Problem~24]{PS}), the fact that
the series $G(z)$ is not a rational function of~$z$ implies
that infinitely many terms of the sequence $V_n(x)$,
where $n=1,2,\dots$, do not vanish.
\end{proof}

\begin{lemma} \label{lem:conj}
Let $\bar\mu,\bar\al,\bar\la,q$ be complex numbers with $\bar\al\ne0$
and $\vert q\vert>1$.
Define the sequence $(\bar v_n)_{n\ge0}$ by \eqref{e52} with
$x=\bar\mu$, $\alpha$ replaced by $\bar\al$, and $\la$ replaced by
$\bar\la$, and let $\bar V_n=\det_{0\le i,j\le n-1}(\bar v_{i+j})$
be the corresponding Hankel determinant.
Then we have
$$
\vert \bar V_n\vert\le \vert q\vert^{\frac {2} {3}n^3+o(n^3)},
\quad \quad \text {as }n\to\infty.
$$
\end{lemma}
\begin{proof}
{}From \eqref{e09} with $\al$ replaced by $\bar\al$,
$\mu$ replaced by $\bar\mu$, and with $b_j(q)=q^j-\bar\la$,
we see that
$$
\vert\bar v_n\vert\le \vert q\vert ^{\frac {1} {2}n^2+o(n^2)}.$$
Hence, we have
\begin{align*}
\vert \bar V_n\vert&\le n!\max_{\tau\in \mathfrak S_n}
\prod _{i=0} ^{n-1}\vert  \bar v_{i+\tau(i)}\vert
\le n!\max_{\tau\in \mathfrak S_n}\prod _{i=0} ^{n-1}
{\vert q\vert^{   \frac {(i+\tau(i))^2}2 +o((i+\tau(i))^2)}}\\
&\le n!
\prod _{i=0} ^{n-1}
{\vert q\vert^{   \frac {(2i)^2}2 +o(n^2)}}
\le {\vert q\vert^{ \frac {2} {3}n^3+o(n^3) }},
\end{align*}
as desired.
\end{proof}

We are now finally in the position to prove Theorems~\textup{\ref{t1}}
and~\textup{\ref{t2}}. Our proof simplifies the $p$-adic approach of
B\'ezivin~\cite{Be} and Choulet~\cite{Ch}.

\begin{proof}[Proof of Theorems~\textup{\ref{t1}} and~\textup{\ref{t2}}]
Let $q = \rho/\sigma \in \mathbb{Q}$, $\left\vert q \right\vert > 1$ and $\rho > 1$.
Furthermore, let $\gamma = (\log \rho)/(\log |\sigma|)$
($\gamma = \infty$ if $q \in \mathbb{Z}$). Let us now assume that all the numbers
$\alpha$, $\lambda$ and $\mu = F_q(\alpha; \lambda)$ are algebraic and write
$K = \mathbb{Q}(\alpha, \lambda, \mu)$, and $d = [K:\mathbb{Q}]$.

In considering $V_n$ we write, as before in~\eqref{e31},
$V_n=\Delta_n\wt V_n$ and note that, as $n\to\infty$, we have
\begin{gather}
|V_n|\le|q|^{-An^3+o(n^3)},
\label{eq:Vn}
\\
\label{eq:Deltan}
\deg_q\Delta_n(q)=Bn^3+o(n^3), \quad \quad
|\Delta_n(q)|=\vert q\vert^{Bn^3+o(n^3)},
\end{gather}
where
\begin{equation}
\begin{aligned}
A=\frac12, \quad B=\frac{65}{216}-\frac{\Im\Li_2(e^{2\pi
\sqrt{-1}/3})}{\pi^2\sqrt3}
\qquad &\text{if $\lambda=0$}, \\
A=\frac13, \quad B=\frac7{27}-\frac{\Im\Li_2(e^{2\pi
\sqrt{-1}/3})}{\pi^2\sqrt3}
\qquad &\text{if $\lambda\ne0$},
\end{aligned}
\label{e40}
\end{equation}
by Propositions~\ref{IRP3} and \ref{IRP5}.
On the other hand, by Lemma~\ref{lem:conj},
for all $K$-conjugates $V^{[i]}_n$ of $V_n$ we have
$$
|V^{[i]}_n|\le |q|^{Cn^3+o(n^3)},\quad \quad i=1,2,\dots,d,
$$
where $C=2/3$. (Of course, for $i=1$, that is, the case where
$V^{[i]}_n=V_n$, we have a better estimate in \eqref{eq:Vn}.)
Clearly, $\Delta_n(q)$ remains invariant under conjugation, whence,
by \eqref{eq:Vn} and \eqref{eq:Deltan}, we have
\begin{equation} \label{eq:V1n}
\vert\wt V_n\vert=\vert\wt V^{[1]}_n\vert\le
\vert q\vert ^{-(A+B)n^3+o(n^3)}
\quad \quad \text {as }n\to\infty,
\end{equation}
and, for $i=2,3,\dots,d$,
\begin{equation} \label{eq:Vin}
\vert\wt V^{[i]}_n\vert\le \vert q\vert ^{(C-B)n^3+o(n^3)}
\quad \quad \text {as }n\to\infty.
\end{equation}

We know that $V_n$ is a polynomial in $q$, $\alpha$, $\lambda$ and
$\mu$ with integer coefficients, hence also $\wt V_n$.
Since the degree of $V_n$ in each of $\mu$, $\la$, $\al$ is at most
$n^2$ (see the paragraph containing \eqref{e10}),
the same is also true for $\wt V_n$. On the other hand, by
\eqref{e10}, we know that
the degree in $q$ of $V_n$ is at most $2n^3/3+o(n^3)$, whence we are able
to find
a positive integer $\Omega(n)$, $\log \Omega(n) = o(n^3)$, such that
\begin{equation} \label{eq:ZK}
\sigma^{(C-B)n^3}\Omega(n)\wt V_n\in\mathbb Z_K,
\end{equation}
where $\mathbb Z_K$ denotes
the ring of integers of $K$. If $V_n\ne0$, then the product of all
$K$-conjugates of the $K$-integer in \eqref{eq:ZK} is a non-zero
integer. Therefore, using \eqref{eq:V1n} and \eqref{eq:Vin},
\begin{align*}
1
&\le \bigg\vert\prod _{i=1} ^{d}\sigma^{(C-B)n^3}\Omega(n)\wt V^{[i]}_n
\bigg\vert\\
&\le|\sigma|^{(C-B)dn^3}\exp(o(n^3))\,\vert \wt V_n\vert
\prod_{i=2}^d|\wt{V}^{[i]}_n|
\\
&\le|\sigma|^{(C-B)dn^3}|q|^{-(A+B)n^3+(C-B)(d-1)n^3}\exp(o(n^3))
\\
&\le|\sigma|^{(A+C)n^3}\rho^{-(A+C-d(C-B))n^3+o(n^3)}
\\
&\le\rho^{\{(A+C)/\gamma-(A+C-d(C-B))\}n^3+o(n^3)}.
\end{align*}
If
\begin{equation}
\frac1\gamma(A+C)-\bigl(A+C-d(C-B)\bigr)<0,
\label{e38}
\end{equation}
then the above inequality implies that $V_n=0$ for all large $n$,
contradicting Lemma~\ref{lem:Kron}.
The reader should note that \eqref{e38} can only hold if
\begin{equation}
A+C-d(C-B)>0
\quad\text{or, equivalently,}\quad
d<\frac{A+C}{C-B},
\label{e37b}
\end{equation}
in which case, we have
\begin{equation} \label{eq:gamm}
\gamma>\frac{A+C}{A+C-d(C-B)}.
\end{equation}

{}From~\eqref{e40} it follows that the only values of the degree $d=[K:\mathbb{Q}]$
satisfying~\eqref{e37b} are $d=1$ and $d=2$. Theorems~\ref{t1} and \ref{t2}
follow then from~\eqref{eq:gamm} with $d=2$ (in Theorem~\ref{t1})
and $d=1$ (Theorem~\ref{t2}) by using the values of $A$ and $B$
from~\eqref{e40}, and $C=2/3$.
\end{proof}

\begin{remark}
\label{rem8}
We have a strong feeling that the method used in this work potentially
makes it possible to deduce irrationality measures for the values
of~$F_q(\alpha;\lambda)$ in the cases when the number in question is
irrational by Theorem~\ref{t2}. The only problem, which we are not able
to overcome, is to establish the required density
of non-vanishing of the determinant $V_n(x)$
in~\eqref{e51} for a given~$x$. More precisely, in our proof of
Theorems~\ref{t1} and~\ref{t2} we use the fact
(see Lemma~\ref{lem:Kron}) that
$V_n(x)\ne0$ infinitely often, and this is
(more than) sufficient for a quantitative
irrationality, respectively, non-quadraticity result.
We expect that a stronger assertion is true, which would then indeed yield
irrationality measures for values of~$F_q(\alpha;\lambda)$.
Namely, for a given $x\in\mathbb C$ and the sequence $v_n(x)$
defined in~\eqref{e52}, there should exist two positive constants $c_1$ and $c_2$, $c_1<c_2$,
such that for any $m\ge1$ one can find an index~$n$ in the range $c_1m<n<c_2m$,
for which the Hankel determinant $V_n(x)$ in~\eqref{e51} does not vanish.
(In fact, we need this statement only for rational values of~$x$,
but this does not seem to be easier than the general case.) The belief
in this statement rests upon the fact that the sequence $v_n(x)$
is `highly structured' (for instance, it is a solution of
the simple recurrence relation~\eqref{e08} or~\eqref{e52}
with general~$x$; cf.~\cite{Be} and~\cite{Ch});
hence $V_n(x)$~should admit a certain structure as well.
In fact, it was pointed out by the anonymous referee that
in the case $\lambda = 0$ of the Tschakaloff function,
$V_n = V_n(\mu)$~is nonzero for all~$n$ if $q > 1$ and $\alpha > 0$.
This follows from Lemma~2.2 in~\cite{Be}, which provides
in this case the expression
$$
V_n=\alpha^{n^2-n}\sum_{1 \leq j_1 < \dots < j_n}
w_{j_1} \dotsb w_{j_n} \Bigl(\frac{\alpha}{q} \Bigr)^{j_1 + \dots + j_n}
\bigl( V(s_{j_1}, \dots, s_{j_n}) \bigr)^2,
$$
where $w_j = q^{-j(j+1)/2}$, $s_j = q^{-j}$, and $V(s_{j_1}, \dots, s_{j_n})$ is
the Vandermonde determinant built on $s_{j_1}, \dots, s_{j_n}$.
The proof of this result is based on the tail expression~\eqref{e23}
of $v_n = v_n(\mu)$ and does not work for $V_n(x)$ with $x$~general.
This fact clearly supports our expectations above, although it is not enough
for irrationality measures.
\end{remark}

\section*{Acknowledgments}
The authors are indebted to the anonymous referee for pointing out some subtle
insufficiencies in an earlier version of the paper, and for several
helpful comments concerning the discussion in Remark~\ref{rem8}.

\end{document}